\documentclass[12pt]{article}
\usepackage[a4paper,margin=1in]{geometry}
\usepackage{amssymb}
\usepackage{amsmath}
\usepackage{amsfonts}
\usepackage{amsthm}
\usepackage{color}
\usepackage{float}
\usepackage[all]{xy}
\usepackage{scalerel}
\usepackage{mathabx}
\usepackage{delimset}
\usepackage{mathtools}
\usepackage{hyperref}

\def\multiset#1#2{\ensuremath{\left(\kern-.3em\left(\genfrac{}{}{0pt}{}{#1}{#2}\right)\kern-.3em\right)}}

\def\P{\mathop{\mbox{\textup{P}}}\nolimits}

\newcommand{\N}{\mathbb{N}}

\newcommand{\R}{\mathbb{R}}
\newcommand{\T}{\mathrm{T}}

\newcommand{\LL}{\mathrm{L}}
\newcommand{\F}{\mathrm{F}}

\newcommand{\J}{\mathrm{J}}
\newcommand{\M}{\mathrm{M}}
\newcommand{\U}{\mathrm{U}}

\newcommand{\fibonomial}{\genfrac{\{}{\}}{0pt}{}}
\newcommand{\qbinom}{\genfrac{[}{]}{0pt}{}}

\newtheorem{theorem}{Theorem}

\newtheorem{definition}{Definition}

\newtheorem{corollary}{Corollary}

\begin{document}

\title{\textbf{Simplicial $d$-Polytopic Numbers Defined on Generalized Fibonacci Polynomials}}
\author{Ronald Orozco L\'opez}
\newcommand{\Addresses}{{
  \bigskip
  \footnotesize

  \textit{E-mail address}, R.~Orozco: \texttt{rj.orozco@uniandes.edu.co}
  
}}

\maketitle

\begin{abstract}
In this article, we introduce the simplicial $d$-polytopic numbers defined on generalized Fibonacci polynomials. We establish basic identities and find $q$-identities known. Furthermore, we find generating functions for the simplicial $d$-polytopic numbers and for the squares of the generalized triangular numbers. Finally, we compute sums of reciprocals of generalized Fibonacci polynomials and generalized triangular numbers. Here we introduce the Zeta function defined on generalized Fibonacci polynomials. 
\end{abstract}


\section{Introduction}
There is a growing research on analogous sequences of numbers defined on generalized Fibonacci polynomials: Catalan numbers \cite{bennet,ekhad}, Bernoulli and Euler polynomials \cite{kus}, Eulerian numbers \cite{ruiz}, among others. In this paper, we define simplicial $d$-polytopic numbers on generalized Fibonacci polynomials. The advantage of doing so is that we obtain for free analogues on Fibonacci sequences, Pell, Jacobthal, Mersenne, among others.

The simplicial polytopic numbers \cite{deza} are a family of sequences of figurate numbers corresponding to the $d$-dimensional simplex for each dimension $d$, where $d$ is a non-negative integer. For $d$ ranging from $1$ to $5$, we have the following simplicial polytopic numbers, respectively: non-negative numbers $\N$, triangular numbers $\mathrm{T}$, tetrahedral numbers $\mathrm{Te}$, pentachoron numbers $\mathrm{P}$ and hexateron numbers $\mathrm{H}$. A list of the above sets of numbers is as follows:
\begin{align*}
    \N&=(0,1,2,3,4,5,6,7,8,9,\ldots),\\
    \mathrm{T}&=(0,1,3,6,10,15,21,28,36,45,55,66,\ldots),\\
    \mathrm{Te}&=(0,1,4,10,20,35,56,84,120,165,\ldots),\\
    \mathrm{P}&=(0,1,5,15,35,70,126,210,330,495,715,\ldots),\\
    \mathrm{H}&=(0,1,6,21,56,126,252,462,792,1287,\ldots).
\end{align*}
The $n^{th}$ simplicial $d$-polytopic numbers $\P_{n}^{d}$ are given by the formulae
\begin{equation*}
    \P_{n}^{d}=\binom{n+d-1}{d}=\frac{n^{(d)}}{d!},
\end{equation*}
where $x^{(d)}=x(x+1)(x+2)\cdots(x+d-1)$ is the rising factorial.
The generating function of the simplicial $d$-polytopic numbers is
\begin{equation}\label{eqn_gf_spn}
    \sum_{n=1}^{\infty}\binom{n+d-1}{d}x^{n}=\frac{x}{(1-x)^{d+1}}.
\end{equation}
In this paper, the $n$-th generalized simplicial $d$-polytopic number is defined by
    \begin{equation*}
        \fibonomial{n+d-1}{d}_{s,t}=\frac{\brk[c]{n}_{s,t}\brk[c]{n+1}_{s,t}\cdots\brk[c]{n+d-1}_{s,t}}{\brk[c]{d}_{s,t}!},
    \end{equation*}
where $\{n\}_{s,t}$ are the generalized Fibonacci polynomials. We will establish basic identities for generalized simplicial $d$-polytopic numbers, in particular for generalized triangular and tetrahedral numbers. These sequences are part of the Encyclopedia of Integer Sequences \cite{sloane}. Some known $q$-identities are found, \cite{sch,war}. We establish generating functions for the generalized simplicial $d$-polytopic numbers $\fibonomial{n+d-1}{d}_{s,t}$ and for the sequence $\fibonomial{n+1}{2}_{s,t}^2$. Finally, we introduce the $(s,t)$-Zeta function $\zeta_{s,t}(z)$ and find some values for $\zeta_{s,t}(1)$. In addition, we calculate reciprocal sums of generalized Fibonacci polynomials and generalized triangular numbers.

\section{Preliminaries}

The generalized Fibonacci numbers on the parameters $s,t$ are defined by
\begin{equation}\label{eqn_def_fibo}
    \brk[c]{n+2}_{s,t}=s\{n+1\}_{s,t}+t\{n\}_{s,t}
\end{equation}
with initial values $\brk[c]{0}_{s,t}=0$ and $\brk[c]{1}_{s,t}=1$, where $s\neq0$ and $t\neq0$. In \cite{Sa1} this sequence is called the generalized Lucas sequence.
Below are some important specializations of generalized Fibonacci and Lucas numbers.
\begin{enumerate}
    \item When $s=2,t=-1$, then $\brk[c]{n}_{2,-1}=n$, the positive integer.
    \item When $s=1,t=1$, then $\brk[c]{n}_{1,1}=F_n$, the Fibonacci numbers
    \begin{equation*}
        \F_{n}=(0,1,1,2,3,5,8,13,21,34,55,\ldots).
    \end{equation*}
    \item When $s=2,t=1$, then $\brk[c]{n}_{2,1}=P_n$, where $P_n$ are the Pell numbers
    \begin{equation*}
        \P_n=(0,1,2,5,12,29,70,169,408\ldots).  
    \end{equation*}
    \item When $s=1,t=2$, then $\brk[c]{n}_{1,2}=J_n$, where $J_n$ are the Jacosbthal numbers
    \begin{equation*}
        \J_n=(0,1,1,3,5,11,21,43,85,171,\ldots).
    \end{equation*}
    \item When $s=3,t=-2$, then $\brk[c]{n}_{3,-2}=M_n$, where $M_n=2^n-1$ are the Mersenne numbers
    \begin{equation*}
        \M_n=(0,1,3,7,15,31,63,127,255,\ldots).
    \end{equation*}
    \item When $s=p+q,t=-pq$, then $\brk[c]{n}_{p+q,-pq}=\brk[s]{n}_{p,q}$, where $\brk[s]{n}_{p,q}$ are the $(p,q)$-numbers
    \begin{equation*}
        \brk[s]{n}_{p,q}=(0,1,\brk[s]{2}_{p,q},[3]_{p,q},[4]_{p,q},[5]_{p,q},[6]_{p,q},[7]_{p,q},[8]_{p,q}\ldots).
    \end{equation*}
    When $p=1$, we obtain the $q$-numbers $\brk[s]{n}_{q}=1+q+q^2+q^3+\cdots+q^{n-1}$.
    \item When $s=2t,t=-1$, then $\brk[c]{n}_{2t,-1}=\U_{n-1}(t)$, where $\U_n(t)$ are the Chebysheff polynomials of the second kind, with $\U_{-1}(t)=0$.
\end{enumerate}
The $(s,t)$-Fibonacci constant is the ratio toward which adjacent $(s,t)$-Fibonacci polynomials tend. This is the only positive root of $x^{2}-sx-t=0$. We will let $\varphi_{s,t}$ denote this constant, where
\begin{equation*}
    \varphi_{s,t}=\frac{s+\sqrt{s^{2}+4t}}{2}
\end{equation*}
and
\begin{equation*}
    \varphi_{s,t}^{\prime}=s-\varphi_{s,t}=-\frac{t}{\varphi_{s,t}}=\frac{s-\sqrt{s^{2}+4t}}{2}.
\end{equation*}
Some specializations of the constants $\varphi_{s,t}$ and $\varphi_{s,t}^\prime$ are:
\begin{enumerate}
    \item When $s=2$ and $t=-1$, then $\varphi_{2,-1}=1$ and $\varphi_{2,-1}^\prime=1$.
    \item When $s=1$ and $t=1$, then $\varphi_{1,1}=\varphi=\frac{1+\sqrt{5}}{2}$ and $\varphi_{1,1}^\prime=\varphi^\prime=\frac{1-\sqrt{5}}{2}$.
    \item When $s=2$ and $t=1$, then $\varphi_{2,1}=1+\sqrt{2}$ and $\varphi_{2,1}^\prime=1-\sqrt{2}$.
    \item When $s=1$ and $t=2$, then $\varphi_{1,2}=2$ and $\varphi_{1,2}^\prime=-1$.
    \item When $s=3$ and $t=-2$, then $\varphi_{3,-2}=2$ and $\varphi_{3,-2}^\prime=1$.
    \item When $s=p+q$ and $t=-pq$, then $\varphi_{p+q,-pq}=p$ and $\varphi_{p+q,-pq}^\prime=q$.
    \item When $s=2t$ and $t=-1$, then $\varphi_{2t,-1}=\frac{t+\sqrt{t^2-1}}{2}$ and $\varphi_{2t,-1}^\prime=\frac{t-\sqrt{t^2-1}}{2}$.
\end{enumerate}

The $(s,t)$-fibonomial coefficients are define by
\begin{equation*}
    \fibonomial{n}{k}_{s,t}=\frac{\brk[c]{n}_{s,t}!}{\brk[c]{k}_{s,t}!\brk[c]{n-k}_{s,t}!},
\end{equation*}
where $\brk[c]{n}_{s,t}!=\brk[c]{1}_{s,t}\brk[c]{2}_{s,t}\cdots\brk[c]{n}_{s,t}$ is the $(s,t)$-factorial or generalized fibotorial.

The $(s,t)$-fibonomial coefficients satisfy the following Pascal recurrence relationships. For $1\leq k\leq n-1$ it holds that
\begin{align}
    \fibonomial{n+1}{k}_{s,t}&=\varphi_{s,t}^{k}\fibonomial{n}{k}_{s,t}+\varphi_{s,t}^{\prime(n+1-k)}\fibonomial{n}{k-1}_{s,t},\label{eqn_pascal1}\\
    &=\varphi_{s,t}^{\prime(k)}\fibonomial{n}{k}_{s,t}+\varphi_{s,t}^{n+1-k}\fibonomial{n}{k-1}_{s,t}.\label{eqn_pascal2}
\end{align}

Set $s,t\in\R$, $s\neq0$, $t\neq0$. If $s^2+4t\neq0$, define the $(\varphi,\varphi^\prime)$-derivative $D_{\varphi,\varphi^\prime}$ of the function $f(x)$ as
\begin{equation}
(D_{\varphi,\varphi^\prime}f)(x)=
\begin{cases}
\frac{f(\varphi_{s,t}x)-f(\varphi_{s,t}^{\prime}x)}{(\varphi_{s,t}-\varphi_{s,t}^{\prime})x},&\text{ if }x\neq0;\\
f^{\prime}(0),&\text{ if }x=0
\end{cases}
\end{equation}
provided $f(x)$ differentiable at $x=0$. If $s^2+4t=0$, $t<0$, define the $(\pm i\sqrt{t},\pm i\sqrt{t})$-derivative of the function $f(x)$ as
\begin{equation}
    (D_{\pm i\sqrt{t},\pm i\sqrt{t}}f)(x)=f^{\prime}(\pm i\sqrt{t}x).
\end{equation}
The $(\varphi,\varphi^\prime)$-derivative $D_{\varphi,\varphi^\prime}$ fulfills the following properties: linearity
        \begin{align*}
            D_{\varphi,\varphi^\prime}(\alpha f+\beta g)&=\alpha D_{\varphi,\varphi^\prime}f+\beta D_{\varphi,\varphi^\prime}g,
        \end{align*}
product rules
        \begin{equation*}
            D_{\varphi,\varphi^\prime}(f(x)g(x))=f(\varphi_{s,t}x)D_{\varphi,\varphi^\prime}g(x)+g(\varphi_{s,t}^\prime x)D_{\varphi,\varphi^\prime}f(x),
        \end{equation*}
         \begin{equation*}
             D_{\varphi,\varphi^\prime}(f(x)g(x))=f(\varphi_{s,t}^{\prime}x)D_{\varphi,\varphi^\prime}g(x)+g(\varphi_{s,t}x)D_{\varphi,\varphi^\prime}f(x),
         \end{equation*}
and quotient rules
         \begin{equation*}
             D_{\varphi,\varphi^\prime}\left(\frac{f(x)}{g(x)}\right)=\frac{g(\varphi_{s,t}x)D_{\varphi,\varphi^\prime}f(x)-f(\varphi_{s,t}x)D_{\varphi,\varphi^\prime}g(x)}{g(\varphi_{s,t}x)g(\varphi_{s,t}^{\prime}x)},
         \end{equation*}
         \begin{equation*}
             D_{\varphi,\varphi^\prime}\left(\frac{f(x)}{g(x)}\right)=\frac{g(\varphi_{s,t}^{\prime}x)D_{\varphi,\varphi^\prime}f(x)-f(\varphi_{s,t}^{\prime}x)D_{\varphi,\varphi^\prime}g(x)}{g(\varphi_{s,t}x)g(\varphi_{s,t}^{\prime}x)}.
         \end{equation*}
Define the $n$-th $(\varphi,\varphi^\prime)$-derivative of the function $f(x)$ recursively as
\begin{equation*}
    D_{\varphi,\varphi^\prime}^nf(x)=D_{\varphi,\varphi^\prime}(D_{\varphi,\varphi^\prime}^{n}f(x)).
\end{equation*}

\section{Generalized simplicial d-polytopic numbers}

\subsection{Definition and basic properties}
\begin{definition}
The $n$-th generalized simplicial $d$-polytopic number is defined by
    \begin{equation}
        \fibonomial{n+d-1}{d}_{s,t}=\frac{\brk[c]{n}_{s,t}\brk[c]{n+1}_{s,t}\cdots\brk[c]{n+d-1}_{s,t}}{\brk[c]{d}_{s,t}!}.
    \end{equation}
The $(s,t)$-analog of the triangular numbers, tetrahedral numbers, pentachoron numbers and hexateron numbers are:
\begin{align*}
    \mathrm{T}_{s,t}&=\Bigg\{\mathrm{T}_{n}(s,t)=\fibonomial{n+1}{2}_{s,t}:n\geq0\Bigg\},\\
    \mathrm{Te}_{s,t}&=\Bigg\{\mathrm{Te}_{n}(s,t)=\fibonomial{n+2}{3}_{s,t}:n\geq0\Bigg\},\\
    \mathrm{P}_{s,t}&=\Bigg\{\mathrm{P}_{n}(s,t)=\fibonomial{n+3}{4}_{s,t}:n\geq0\Bigg\},\\
    \mathrm{H}_{s,t}&=\Bigg\{\mathrm{H}_{n}(s,t)=\fibonomial{n+4}{5}_{s,t}:n\geq0\Bigg\}.
\end{align*}
\end{definition}
From Pascal recurrence Eqs(\ref{eqn_pascal1}) and (\ref{eqn_pascal2})
\begin{align}
    \fibonomial{n+d}{d}_{s,t}&=\varphi_{s,t}^{d}\fibonomial{n+d-1}{d}_{s,t}+\varphi_{s,t}^{\prime n}\fibonomial{n+d-1}{d-1}_{s,t},\label{eqn_gspn_recu1}\\
    &=\varphi_{s,t}^{\prime d}\fibonomial{n+d-1}{d}_{s,t}+\varphi_{s,t}^{n}\fibonomial{n-d-1}{d-1}_{s,t}\label{eqn_gspn_recu2}.
\end{align}
It is a well-known fact that the sum of the first $n$ terms of a sequence of $d$-polytopic numbers is the $n$-th term of a sequence of $(d+1)$-polytopic numbers, i.e,
\begin{equation}
    \sum_{k=1}^{n}\P_{n}^d=\P_{n}^{d+1}.
\end{equation}
We then obtain the $(s,t)$-analog of the above formula.
\begin{theorem}\label{theo_reduc}
For all $n\geq1$, the sum of generalized $d$-polytopic numbers is
    \begin{align}
        \fibonomial{n+d}{d+1}_{s,t}&=\sum_{k=1}^{n}\varphi_{s,t}^{(d+1)(n-k)}\varphi_{s,t}^{\prime(k-1)}\fibonomial{k+d-1}{d}_{s,t},\\
        \fibonomial{n+d}{d+1}_{s,t}&=\sum_{k=1}^{n}\varphi_{s,t}^{\prime(d+1)(n-k)}\varphi_{s,t}^{k-1}\fibonomial{k+d-1}{d}_{s,t}.
    \end{align}
\end{theorem}
\begin{proof}
The proof is by induction on $n$. When $n=1$, $\fibonomial{d+1}{d+1}_{s,t}=\fibonomial{d}{d}_{s,t}$. For $n=2$, it follows that
\begin{align*}
    \fibonomial{d+2}{d+1}_{s,t}&=\varphi_{s,t}^{d+1}\fibonomial{d}{d}_{s,t}+\varphi_{s,t}^\prime\fibonomial{d+1}{d}_{s,t}.
\end{align*}
Suppose the statement is true for $n$ and let's prove it for $n+1$. We have
\begin{align*}
    \fibonomial{n+d+1}{d+1}_{s,t}&=\varphi_{s,t}^{d+1}\fibonomial{n+d}{d+1}_{s,t}+\varphi_{s,t}^{\prime n}\fibonomial{n+d}{d}_{s,t}\\
    &=\varphi_{s,t}^{d+1}\sum_{k=1}^{n}\varphi_{s,t}^{(d+1)(n-k)}\varphi_{s,t}^{\prime(k-1)}\fibonomial{k+d-1}{d}_{s,t}+\varphi_{s,t}^{\prime n}\fibonomial{n+d}{d}_{s,t}\\
    &=\sum_{k=1}^{n}\varphi_{s,t}^{(d+1)(n+1-k)}\varphi_{s,t}^{\prime(k-1)}\fibonomial{k+d-1}{d}_{s,t}+\varphi_{s,t}^{\prime n}\fibonomial{n+d}{d}_{s,t}\\
    &=\sum_{k=1}^{n+1}\varphi_{s,t}^{(d+1)(n+1-k)}\varphi_{s,t}^{\prime(k-1)}\fibonomial{k+d-1}{d}_{s,t}.
\end{align*}
The proof is reached.
\end{proof}
From Theorem \ref{theo_reduc} we obtain the following known result about $q$-binomial coefficients:
\begin{equation*}
    \qbinom{n+d}{d+1}_{q}=\sum_{k=1}^{n}q^{k-1}\qbinom{k+d-1}{d}_{q}=\sum_{k=1}^{n}q^{(d+1)(n-k)}\qbinom{k+d-1}{d}_{q}.
\end{equation*}

\section{Some specializations}

\subsection{Generalized triangular numbers}

Some specializations of generalized triangular numbers are
\begin{align*}
    \fibonomial{n+1}{2}_{1,1}&=F_{n}F_{n+1}\\
    &=(0,1,2,6,15,40,104,273,\ldots),\\
    \fibonomial{n+1}{2}_{2,1}&=\frac{1}{2}P_{n}P_{n+1}\\
    &=(0,1,5,30,174,1015,5915,\ldots),\\
    \fibonomial{n+1}{2}_{1,2}&=J_{n}J_{n+1}\\
    &=(0,1,2,6,15,55,231,903,3655,\ldots),\\
    \fibonomial{n+1}{2}_{3,-2}&=\frac{1}{3}(2^n-1)(2^{n+1}-1)\\
    &=(0,1,7,35,155,651,2667,10795,43435,174251,\ldots).
\end{align*}
The $\fibonomial{n+1}{2}_{1,1}$ numbers are known as Golden rectangle numbers, A001654. The $\fibonomial{n+1}{2}_{2,1}$ numbers may be called Pell triangles, A084158. The $\fibonomial{n+1}{2}_{1,2}$ numbers are known as Jacobsthal oblong numbers, A084175. The $\fibonomial{n+1}{2}_{3,-2}$ numbers are the Gaussian binomial coefficients $\qbinom{n}{2}_{q}$ for $q=2$, A006095.

From Eqs. (\ref{eqn_gspn_recu1}) and (\ref{eqn_gspn_recu2}), and Theorem \ref{theo_reduc},
\begin{align*}
    \fibonomial{n+2}{2}_{s,t}&=\varphi_{s,t}^2\fibonomial{n+1}{2}+\varphi_{s,t}^{\prime n}\brk[c]{n+1}_{s,t},\\
    \fibonomial{n+2}{2}_{s,t}&=\varphi_{s,t}^{\prime2}\fibonomial{n+1}{2}_{s,t}+\varphi_{s,t}^{n}\brk[c]{n+1}_{s,t},\\
    \fibonomial{n+1}{2}_{s,t}&=\sum_{k=1}^{n}\varphi_{s,t}^{2(n-k)}\varphi_{s,t}^{\prime(k-1)}\brk[c]{k}_{s,t},\\
    \fibonomial{n+1}{2}_{s,t}&=\sum_{k=1}^{n}\varphi_{s,t}^{\prime2(n-k)}\varphi_{s,t}^{k-1}\brk[c]{k}_{s,t}.
\end{align*}
From the above identities we obtain the identity of Warnaar \cite{war}
\begin{equation*}
    \qbinom{n+1}{2}_{q}=\sum_{k=1}^{n}\frac{1-q^k}{1-q}q^{2(n-k)}=\sum_{k=1}^{n}q^{k-1}\frac{1-q^k}{1-q}.
\end{equation*}

\begin{theorem}\label{eqn_tri_sq}
For all $n\in\N$,
\begin{enumerate}
    \item The $(s,t)$-analog of the identity $\T_{n+1}-\T_{n}=(n+1)^2$ is
    \begin{equation*}
        \fibonomial{n+2}{2}_{s,t}=t\fibonomial{n+1}{2}_{s,t}+\brk[c]{n+1}^2.
    \end{equation*}
    \item  The $(s,t)$-analog of the identity $\T_{n-1}+\T_{n}=n^2$ is
    \begin{equation*}
        \varphi_{s,t}^{n+1}\fibonomial{n}{2}_{s,t}+\varphi_{s,t}^{\prime(n-1)}\fibonomial{n+1}{2}_{s,t}=\frac{\brk[a]{n}_{s,t}}{\brk[c]{2}_{s,t}}\brk[c]{n}_{s,t}^2,
    \end{equation*}
    where $\brk[a]{n}_{s,t}=\varphi_{s,t}^n+\varphi_{s,t}^{\prime n}$.
    \item The $(s,t)$-analogue of the alternating sum squares $\T_{n}=\sum_{k=1}^{n}(-1)^{n-k}k^2$ is
    \begin{equation*}
        \fibonomial{n+1}{2}_{s,t}=\sum_{k=1}^{n}t^{n-k}\brk[c]{k}_{s,t}^2.
    \end{equation*}
\end{enumerate}
\end{theorem}
\begin{proof}
We have
    \begin{align*}
        \fibonomial{n+2}{2}_{s,t}&=\frac{\brk[c]{n+1}_{s,t}\brk[c]{n+2}_{s,t}}{\brk[c]{2}_{s,t}}\\
        &=\frac{\brk[c]{n+1}_{s,t}(s\brk[c]{n+1}_{s,t}+t\brk[c]{n}_{s,t})}{\brk[c]{2}_{s,t}}\\
        &=s\frac{\brk[c]{n+1}_{s,t}^2}{s}+t\frac{\brk[c]{n+1}_{s,t}\brk[c]{n}_{s,t}}{\brk[c]{2}_{s,t}}\\
        &=\brk[c]{n+1}_{s,t}^2+t\fibonomial{n+1}{2}_{s,t}.
    \end{align*}
The statement 2 is proved as follows:
\begin{align*}
    \varphi_{s,t}^{n+1}\fibonomial{n}{2}_{s,t}+\varphi_{s,t}^{\prime(n-1)}\fibonomial{n+1}{2}_{s,t}&=\varphi_{s,t}^{n+1}\frac{\brk[c]{n}_{s,t}\brk[c]{n-1}_{s,t}}{\brk[c]{2}_{s,t}}+\varphi_{s,t}^{\prime(n-1)}\frac{\brk[c]{n}_{s,t}\brk[c]{n+1}_{s,t}}{\brk[c]{2}_{s,t}}\\
    &=\frac{\brk[c]{n}_{s,t}}{\brk[c]{2}_{s,t}}(\varphi_{s,t}^{n+1}\brk[c]{n-1}_{s,t}+\varphi_{s,t}^{\prime(n-1)})\brk[c]{n+1}_{s,t})\\
    &=\frac{\brk[c]{n}_{s,t}}{\brk[c]{2}_{s,t}}\brk[c]{2n}_{s,t}\\
    &=\frac{\brk[a]{n}_{s,t}}{\brk[c]{2}_{s,t}}\brk[c]{n}_{s,t}^2.
\end{align*}
By iterating 1., we obtain the result 3.
\end{proof}
If we set $s=1+q$, $t=-q$ in the above theorem, then
\begin{align*}
   \qbinom{n+2}{2}_{q}&=-q\qbinom{n+1}{2}_{q}+\left(\frac{1-q^{n+1}}{1-q}\right)^2,\\
   \qbinom{n}{2}_{q}+q^{n-1}\qbinom{n+1}{2}_{q}&=\frac{1+q^n}{1+q}\brk[s]{n}_{q}^2,
\end{align*}
and by iterating we get a Schlosser result \cite{sch}
\begin{equation}
    \qbinom{n+1}{2}_{q}=\sum_{k=1}^{n}(-q)^{n-k}\left(\frac{1-q^k}{1-q}\right)^2.
\end{equation}

\begin{theorem}
For all $n\in\N$
\begin{equation}\label{eqn_tri_cube}
    \fibonomial{n+2}{2}_{s,t}^2-t^2\fibonomial{n+1}{2}_{s,t}^2=\left(\frac{\brk[c]{n+2}_{s,t}+t\brk[c]{n}_{s,t}}{s}\right)\brk[c]{n+1}_{s,t}^3.
\end{equation}
\end{theorem}
\begin{proof}
From Theorem \ref{eqn_tri_sq}
\begin{align*}
    \fibonomial{n+2}{2}_{s,t}^2&=\left(t\fibonomial{n+1}{2}_{s,t}+\brk[c]{n+1}_{s,t}^2\right)^2\\
    &=t^2\fibonomial{n+1}{2}_{s,t}^2+2t\fibonomial{n+1}{2}_{s,t}^2\brk[c]{n+1}_{s,t}^2+\brk[c]{n+1}_{s,t}^4\\
    &=t^2\fibonomial{n+1}{2}_{s,t}^2+2t\frac{\brk[c]{n}_{s,t}\brk[c]{n+1}_{s,t}^3}{\brk[c]{2}_{s,t}}+\brk[c]{n+1}_{s,t}^4\\
    &=t^2\fibonomial{n+1}{2}_{s,t}^2+\left(\frac{2t}{s}\brk[c]{n}_{s,t}+\brk[c]{n+1}_{s,t}\right)\brk[c]{n+1}_{s,t}^3\\
    &=t^2\fibonomial{n+1}{2}_{s,t}^2+\left(\frac{t\brk[c]{n}_{s,t}+\brk[c]{n+2}_{s,t}}{s}\right)\brk[c]{n+1}_{s,t}^3.
\end{align*}
The rest of the proof follows easily.
\end{proof}
Eq.(\ref{eqn_tri_cube}) is the $(s,t)$-analog of the identity
\begin{equation*}
    \mathrm{T}_{n+1}^2-\mathrm{T}_{n}^2=n^3.
\end{equation*}
From Warnaar \cite{war},
\begin{equation*}
    \qbinom{n+2}{2}_{q}^2-q^2\qbinom{n+1}{2}_{q}^2=\left(\frac{1-q^{2(n+1)}}{1-q^2}\right)\left(\frac{1-q^{n+1}}{1-q}\right)^2.
\end{equation*}

\begin{theorem}
For all $n\in\N$,
\begin{equation}\label{eqn_cub_sq}
    \sum_{k=1}^{n}t^{2(n-k)}\left(\frac{\brk[c]{k+1}_{s,t}+t\brk[c]{k-1}_{s,t}}{s}\right)\brk[c]{k}_{s,t}^3=\fibonomial{n+1}{2}_{s,t}^2.
\end{equation}
\end{theorem}
\begin{proof}
\begin{align*}
    \fibonomial{n+1}{2}_{s,t}^2&=\left(\fibonomial{n+1}{2}_{s,t}^2-t^2\fibonomial{n}{2}_{s,t}^2\right)+\left(t^2\fibonomial{n}{2}_{s,t}^2-t^4\fibonomial{n-1}{2}_{s,t}^2\right)\\
    &\hspace{2cm}+\cdots+\left(t^{2n-2}\fibonomial{2}{2}_{s,t}^2-t^{2n}\fibonomial{1}{2}_{s,t}^2\right)\\
    &=\left(\frac{\brk[c]{n+1}_{s,t}+t\brk[c]{n-1}_{s,t}}{s}\right)\brk[c]{n}_{s,t}^3+t^2\left(\frac{\brk[c]{n}_{s,t}+t\brk[c]{n-2}_{s,t}}{s}\right)\brk[c]{n-1}_{s,t}^3+\\
    &\cdots+t^{2n-2}\left(\frac{\brk[c]{2}_{s,t}}{s}\right)\brk[c]{1}_{s,t}^3\\
    &=\sum_{k=1}^{n}t^{2(n-k)}\left(\frac{\brk[c]{k+1}_{s,t}+t\brk[c]{k-1}_{s,t}}{s}\right)\brk[c]{k}_{s,t}^3.
\end{align*}    
\end{proof}
The Eq.(\ref{eqn_cub_sq}) can be written as
\begin{equation*}
    \sum_{k=1}^{n}t^{2(n-k)}\left(\frac{\varphi_{s,t}^{2k}-\varphi_{s,t}^{\prime2k}}{\varphi_{s,t}^2-\varphi_{s,t}^{\prime2}}\right)\left(\frac{\varphi_{s,t}^k-\varphi_{s,t}^{\prime k}}{\varphi_{s,t}-\varphi_{s,t}^\prime}\right)^2=\fibonomial{n+1}{2}_{s,t}^2.
\end{equation*}
Eq.(\ref{eqn_cub_sq}) is the $(s,t)$-analog of the identity
\begin{equation*}
    \sum_{k=1}^{n}k^3=\left(\sum_{k=1}^{n}k\right)^2.
\end{equation*}

\begin{corollary}
For all $n\in\N$
\begin{enumerate}
    \item The Fibonacci analog of the sum of cubes:
    \begin{equation*}
        \sum_{k=1}^{n}(F_{k+1}+F_{k-1})F_{k}^3=F_{n}^2F_{n+1}^2.
    \end{equation*}
    \item The Pell analog of the sum of cubes:
    \begin{equation*}
        \sum_{k=1}^{n}(P_{k+1}+P_{k-1})P_{k}^3=\frac{1}{2}P_{n}^2P_{n+1}^2.
    \end{equation*}
    \item The Jacobsthal analog of the sum of cubes:
    \begin{equation*}
        \sum_{k=1}^{n}4^{n-k}(J_{k+1}+2J_{k-1})J_{k}^3=J_{n}^2J_{n+1}^2.
    \end{equation*}
    \item The Mersenne analog of the sum of cubes:
    \begin{equation*}
        \sum_{k=1}^{n}4^{n-k}(2^{k}+1)(2^k-1)^3=\frac{1}{3}(2^n-1)^2(2^{n+1}-1)^2.
    \end{equation*}
    \item The $q$-identity of Warnaar \cite{war}:
    \begin{equation}\label{eqn_war}
    \sum_{k=1}^{n}q^{2(n-k)}\left(\frac{1-q^{2k}}{1-q^2}\right)\left(\frac{1-q^k}{1-q}\right)^2=\qbinom{n+1}{2}_{q}^2.
\end{equation}
\end{enumerate}
\end{corollary}

\subsection{Generalized tetrahedral numbers}
Some specializations of generalized tetrahedral numbers are
\begin{align*}
    \fibonomial{n+2}{3}_{1,1}&=\frac{1}{2}F_{n}F_{n+1}F_{n+2}\\
    &=(0,1,3,15,60,260,1092,4641,19635,\ldots),\\
    \fibonomial{n+2}{3}_{2,1}&=\frac{1}{10}P_{n}P_{n+1}P_{n+2}\\
    &=(1, 12, 174, 2436, 34307, 482664,\ldots),\\
    \fibonomial{n+2}{3}_{1,2}&=\frac{1}{3}J_{n}J_{n+1}J_{n+2}\\
    &=(0,1,5,55,385,3311,25585,208335,\ldots),\\
    \fibonomial{n+2}{3}_{3,-2}&=\frac{1}{21}(2^n-1)(2^{n+1}-1)(2^{n+2}-1)\\
    &=(0,1, 15, 155, 1395, 11811, 97155,\ldots).
\end{align*}
The sequences $\fibonomial{n+2}{3}_{1,1}$, $\fibonomial{n+2}{3}_{2,1}$, and $\fibonomial{n+2}{3}_{3,-2}$ are the sequences A001655, A099930, and A006096, respectively, in \cite{sloane}. The sequence $\fibonomial{n+2}{3}_{1,2}$ has never been investigated.
From Eqs. (\ref{eqn_gspn_recu1}) and (\ref{eqn_gspn_recu2}), and Theorem \ref{theo_reduc}, 
\begin{align*}
    \fibonomial{n+3}{3}_{s,t}&=\varphi_{s,t}^3\fibonomial{n+2}{3}_{s,t}+\varphi_{s,t}^{\prime n}\fibonomial{n+2}{2}_{s,t},\\
    \fibonomial{n+3}{3}_{s,t}&=\varphi_{s,t}^{\prime3}\fibonomial{n+2}{3}_{s,t}+\varphi_{s,t}^{n}\fibonomial{n+2}{2}_{s,t},\\
    \fibonomial{n+2}{3}_{s,t}&=\sum_{k=1}^{n}\varphi_{s,t}^{3(n-k)}\varphi_{s,t}^{\prime(k-1)}\fibonomial{k+1}{2}_{s,t},\\
    \fibonomial{n+3}{3}_{s,t}&=\sum_{k=1}^{n}\varphi_{s,t}^{\prime3(n-k)}\varphi_{s,t}^{k-1}\fibonomial{k+1}{2}_{s,t}.
\end{align*}

\begin{theorem}
For all $n\geq0$
    \begin{equation}\label{eqn_iden_tetra}
        \fibonomial{n+3}{3}_{s,t}=st\fibonomial{n+2}{3}_{s,t}+\brk[c]{n+1}_{s,t}\fibonomial{n+2}{2}_{s,t}.
    \end{equation}
\end{theorem}
\begin{proof}
We have
    \begin{align*}
        \fibonomial{n+3}{3}_{s,t}&=\frac{\brk[c]{n+1}_{s,t}\brk[c]{n+2}_{s,t}\brk[c]{n+3}_{s,t}}{\brk[c]{3}_{s,t}!}\\
        &=\frac{\brk[c]{n+1}_{s,t}\brk[c]{n+2}_{s,t}}{\brk[c]{3}_{s,t}!}(s\brk[c]{n+2}_{s,t}+t\brk[c]{n+1}_{s,t})\\
        &=\frac{\brk[c]{n+1}_{s,t}\brk[c]{n+2}_{s,t}}{\brk[c]{3}_{s,t}!}((s^2+t)\brk[c]{n+1}_{s,t}+st\brk[c]{n}_{s,t})\\
        &=\brk[c]{n+1}_{s,t}\fibonomial{n+2}{2}_{s,t}+st\fibonomial{n+2}{3}_{s,t}.
    \end{align*}
\end{proof}
Eq.(\ref{eqn_iden_tetra}) is the $(s,t)$-analog of the identity
\begin{equation*}
    \mathrm{Te}_{n+1}+2\mathrm{Te}_{n}=(n+1)\mathrm{T}_{n+1}.
\end{equation*}
If we choose $s=1+q$, $t=-q$ in Eq.(\ref{eqn_iden_tetra}), we obtain 
\begin{equation*}
        \qbinom{n+3}{3}_{q}=-(1+q)q\qbinom{n+2}{3}_{q}+\frac{1-q^{n+1}}{1-q}\qbinom{n+2}{2}_{q}.
    \end{equation*}

\section{Generating functions}

\begin{theorem}\label{theo_nder_geo}
For all $n\geq1$,
    \begin{equation}\label{eqn_nder_geo}
        D_{\varphi,\varphi^\prime}^{n}\left(\frac{1}{1-x}\right)=\frac{\brk[c]{n}_{s,t}!}{(\varphi_{s,t}^nx;q)_{n+1}},
    \end{equation}
where $(a;q)_{n}=\prod_{k=0}^{n-1}(1-aq^k)$ is the 
\end{theorem}
\begin{proof}
Note that
\begin{equation}
    D_{\varphi,\varphi^\prime}\left(\frac{1}{1-x}\right)=\frac{1}{(1-\varphi_{s,t}x)(1-\varphi_{s,t}^\prime x)}=\frac{\brk[c]{1}_{s,t}!}{(\varphi_{s,t}x;q)_{2}}.
\end{equation}
Suppose that Eq.(\ref{eqn_nder_geo}) is true for $n$ and let us prove by induction for $n+1$. As
\begin{equation}
    D_{\varphi,\varphi^\prime}(\varphi_{s,t}^nx;q)_{n+1}=-\brk[c]{n+1}_{s,t}(\varphi_{s,t}^n\varphi_{s,t}^\prime x;q)_{n},
\end{equation}
then
\begin{align*}
    D_{\varphi,\varphi^\prime}^{n+1}\left(\frac{1}{1-x}\right)&=D_{\varphi,\varphi^\prime}D_{\varphi,\varphi^\prime}^{n}\left(\frac{1}{1-x}\right)\\
    &=D_{\varphi,\varphi^\prime}\left(\frac{\brk[c]{n}_{s,t}!}{(\varphi_{s,t}^nx;q)_{n+1}}\right)\\
    &=\frac{-\brk[c]{n}_{s,t}!D_{\varphi,\varphi^\prime}(\varphi_{s,t}^nx;q)_{n+1}}{(\varphi_{s,t}^{n+1}x;q)_{n+1}(\varphi_{s,t}^n\varphi_{s,t}^{\prime}x;q)_{n+1}}\\
    &=\frac{\brk[c]{n+1}_{s,t}!(\varphi_{s,t}^{n}\varphi_{s,t}^{\prime}x;q)_{n}}{(\varphi_{s,t}^{n+1}x;q)_{n+1}(\varphi_{s,t}^{n}\varphi_{s,t}^{\prime}x;q)_{n+1}}\\
    &=\frac{\brk[c]{n+1}_{s,t}!}{(1-(\varphi_{s,t}q)^{n+1}x)(\varphi_{s,t}^{n+1}x;q)_{n+1}}\\
    &=\frac{\brk[c]{n+1}_{s,t}!}{(\varphi_{s,t}^{n+1}x;q)_{n+2}}.
\end{align*}
The proof is reached.
\end{proof}

\begin{theorem}
    \begin{equation}
        \sum_{n=1}^{\infty}\fibonomial{n+d-1}{d}_{s,t}x^n=\frac{x}{(\varphi_{s,t}^{d}x;q)_{d+1}}
    \end{equation}
\end{theorem}
\begin{proof}
From Theorem \ref{theo_nder_geo},
    \begin{align*}
        \frac{x}{(\varphi_{s,t}^{d}x;q)_{d+1}}&=\frac{x}{\brk[c]{d}_{s,t}!}D_{\varphi,\varphi^\prime}^d\left(\frac{1}{1-x}\right)\\
        &=\frac{x}{\brk[c]{d}_{s,t}!}D_{\varphi,\varphi^\prime}^{d}\left(\sum_{n=0}^{\infty}x^n\right)\\
        &=\frac{x}{\brk[c]{d}_{s,t}!}\sum_{n=d}^{\infty}\brk[c]{n}_{s,t}\brk[c]{n-1}_{s,t}\cdots\brk[c]{n-d+1}_{s,t}x^{n-d}\\
        &=\frac{x}{\brk[c]{d}_{s,t}!}\sum_{n=0}^{\infty}\frac{\brk[c]{d+n}_{s,t}!}{\brk[c]{n}_{s,t}!}x^n\\
        &=\sum_{n=1}^{\infty}\fibonomial{n+d-1}{d}_{s,t}x^{n}.
    \end{align*}
\end{proof}

\begin{theorem}
    \begin{multline}
        \sum_{n=1}^{\infty}\fibonomial{n+1}{2}_{s,t}^2x^n\\=\frac{x+(\brk[c]{4}\varphi_{s,t}-\varphi_{s,t}^4)x^2-\brk[c]{3}_{s,t}\varphi_{s,t}^3\varphi_{s,t}^{\prime3}x^3+(\brk[c]{3}_{s,t}\varphi_{s,t}^7\varphi_{s,t}^{\prime3}-\brk[c]{4}_{s,t}\varphi_{s,t}^6\varphi_{s,t}^{\prime3})x^4}{(1-t^2x)(1-\varphi_{s,t}^{\prime4}x)(\varphi_{s,t}^4x;q)_{4}}
    \end{multline}
\end{theorem}
\begin{proof}
    \begin{align*}
        &\sum_{n=1}^{\infty}\fibonomial{n+1}{2}_{s,t}^2x^n\\
        &\hspace{1cm}=\sum_{n=1}^{\infty}\sum_{k=1}^{n}t^{2(n-k)}\brk[s]{k}_{\varphi^2,\varphi^{\prime2}}\brk[s]{k}_{\varphi,\varphi^\prime}^2x^n\\
        &\hspace{1cm}=\sum_{k=1}^{\infty}\brk[s]{k}_{\varphi^2,\varphi^{\prime2}}\brk[s]{k}_{\varphi,\varphi^\prime}^2x^k\sum_{n=k}^{\infty}t^{2(n-k)}x^n\\
        &\hspace{1cm}=\sum_{k=1}^{\infty}\brk[s]{k}_{\varphi^2,\varphi^{\prime2}}\brk[s]{k}_{\varphi,\varphi^\prime}^2x^k\sum_{n=0}^{\infty}(t^{2}x)^{n}\\
        &\hspace{1cm}=\frac{1}{1-t^2x}(xD_{\varphi,\varphi^\prime})^2(xD_{\varphi^2,\varphi^{\prime2}})\left\{\frac{x}{1-x}\right\}\\
        &\hspace{1cm}=\frac{x+(\brk[c]{4}_{s,t}\varphi_{s,t}-\varphi_{s,t}^4)x^2-\brk[c]{3}_{s,t}\varphi_{s,t}^3\varphi_{s,t}^{\prime3}x^3+(\brk[c]{3}_{s,t}\varphi_{s,t}^7\varphi_{s,t}^{\prime3}-\brk[c]{4}_{s,t}\varphi_{s,t}^6\varphi_{s,t}^{\prime3})x^4}{(1-t^2x)(1-\varphi_{s,t}^{\prime4}x)(\varphi_{s,t}^4x;q)_{4}}.
    \end{align*}
\end{proof}
The $q$-analog of the above theorem is
\begin{equation*}
        \sum_{n=1}^{\infty}\qbinom{n+1}{2}_{q}^{2}x^n=\frac{x+(\brk[s]{4}_{q}-1)x^2-\brk[s]{3}_{q}q^{3}x^3+(\brk[s]{3}_{q}q^{3}-\brk[s]{4}_{q}q^{3})x^4}{(1-q^2x)(1-q^{4}x)(x;q)_{4}}.
    \end{equation*}

\section{Sum of reciprocals}

The generalized Fibonacci Zeta function, or $(s,t)$-Zeta function, is the function defined by
\begin{equation}
    \zeta_{s,t}(z)=\sum_{n=1}^{\infty}\frac{1}{\brk[c]{n}_{s,t}^z}=1+\frac{1}{s^z}+\frac{1}{(s^2+t)^z}+\frac{1}{(s^3+2st)^z}+\cdots.
\end{equation}
Egami \cite{egami} and Navas \cite{navas} independently studied the zeta function $\zeta_{1,1}(z)$. Landau \cite{landau} studied the problem of evaluating $\zeta_{1,1}(1)$. The function $\zeta_{s,t}(z)$ is convergent when $z>0$ and when either $\varphi_{s,t}>1$ and $0<\vert q\vert<1$ or $\varphi_{s,t}^\prime>1$ and $\vert q\vert>1$. Take $z=\sigma+i\beta$. Then $\zeta_{s,t}(z)$ converge when $\Re(z)>0$. Some specialization of $\zeta_{s,t}(1)$ are:
\begin{align*}
    \zeta_{1,1}(1)=\zeta_{F}(1)&=\sum_{n=1}^{\infty}\frac{1}{F_{n}}=1+1+\frac{1}{2}+\frac{1}{3}+\frac{1}{5}+\frac{1}{8}+\cdots\approx3,359885666243\ldots,\\
    \zeta_{2,1}(1)=\zeta_{P}(1)&=\sum_{n=1}^{\infty}\frac{1}{P_{n}}=1+\frac{1}{2}+\frac{1}{5}+\frac{1}{12}+\frac{1}{29}+\cdots\approx 1,81781609195402\ldots,\\
    \zeta_{1,2}(1)=\zeta_{J}(1)&=\sum_{n=1}^{\infty}\frac{1}{J_{n}}=1+1+\frac{1}{3}+\frac{1}{5}+\frac{1}{11}+\frac{1}{21}+\cdots\approx2,67186147\ldots,\\
    \zeta_{3,-2}(1)=\zeta_{M}(1)&=\sum_{n=1}^{\infty}\frac{1}{M_{n}}=1+\frac{1}{3}+\frac{1}{7}+\frac{1}{15}+\frac{1}{31}+\cdots\approx1,57511520737327\ldots,\\
    \zeta_{2t,-1}(1)=\zeta_{U}(1)&=\sum_{n=1}^{\infty}\frac{1}{U_{n-1}(t)}=1+\frac{1}{2t}+\frac{1}{4t^2-1}+\frac{1}{8t^3-4t}+\frac{1}{16t^4-12t^2+1}+\cdots,
\end{align*}
with $t\neq0,\cos\frac{k\pi}{n+1}$, $k=1,2,\ldots,n$.

\begin{theorem}
    \begin{equation}
        \sum_{n=1}^{\infty}\frac{t^n}{\brk[c]{2n}_{s,t}}=\sqrt{s^2+4t}\left[\LL\left(\varphi_{s,t}^{\prime2}/t\right)-\LL\left(\varphi_{s,t}^{\prime4}/t^2\right)\right],
    \end{equation}
where $\LL(q)=\sum_{n=1}^{\infty}\frac{q^n}{1-q^n}$ is the Lambert function.
\end{theorem}
\begin{proof}
Take into account that
\begin{equation*}
    \frac{t^n}{\brk[c]{2n}_{s,t}}=\sqrt{s^2+4t}\left(\frac{(\varphi_{s,t}^{\prime2}/t)^n}{1-(\varphi_{s,t}^{\prime2}/t)^n}-\frac{(\varphi_{s,t}^{\prime4}/t^2)^n}{1-(\varphi_{s,t}^{\prime4}/t^2)^n}\right)
\end{equation*}
and sum for all $n\geq1$.
\end{proof}

\begin{theorem}
    \begin{equation}
        \sum_{n=1}^{\infty}\frac{1}{\brk[c]{2n-1}_{s,1}}=-\frac{\sqrt{s^2+4}}{4}\theta_{2}(\varphi_{s,1}^{\prime2})^2,
    \end{equation}
where $\theta_{2}(q)=\sum_{n=-\infty}^{\infty}q^{(n+1/2)^2}$.
\end{theorem}
\begin{proof}
\begin{align*}
    \sum_{n=1}^{\infty}\frac{1}{\brk[c]{2n-1}_{s,1}}&=-\frac{\sqrt{s^2+4}}{2}\sum_{n=1}^{\infty}\frac{2(\varphi_{s,1}^{\prime2})^{n-1/2}}{1+(\varphi_{s,1}^{\prime2})^{2n-1}}\\
    &=-\frac{\sqrt{s^2+4}}{4}\sum_{a=-\infty}^{\infty}\frac{2(\varphi_{s,1}^{\prime2})^{a-1/2}}{1+(\varphi_{s,1}^{\prime2})^{2a-1}}\\
    &=-\frac{\sqrt{s^2+4}}{4}\theta_{2}(\varphi_{s,1}^{\prime2})^2.
\end{align*}    
\end{proof}

\begin{theorem}
    \begin{equation}\label{eqn_tri_reci}
        \sum_{n=1}^{\infty}\frac{1}{\brk[c]{n}_{s,t}\brk[c]{n+1}_{s,t}}=\varphi_{s,t}-(1+t)\ln\left(1+\frac{\varphi_{s,t}}{t}\right).
    \end{equation}
where
\begin{equation*}
    \ln_{s,t}(1-x)=-\sum_{n=1}^{\infty}\frac{x^n}{\brk[c]{n}_{s,t}}
\end{equation*}
is an $(s,t)$-analog of function logarithm.
\end{theorem}
\begin{proof}
Suppose that $0<\vert q\vert<1$, $q=\varphi_{s,t}^\prime/\varphi_{s,t}$, and set $a_{n}=1/\brk[c]{n}_{s,t}\brk[c]{n+1}_{s,t}$. As 
\begin{equation*}
    \lim_{n\rightarrow\infty}\frac{a_{n+1}}{a_{n}}=\lim_{n\rightarrow\infty}\frac{\brk[c]{n}_{s,t}}{\brk[c]{n+2}_{s,t}}=\lim_{n\rightarrow\infty}\frac{1-q^n}{\varphi_{s,t}^2-\varphi_{s,t}^{\prime2}q^n}=\frac{1}{\varphi_{s,t}^{2}},
\end{equation*}
then the series in the left-hand of Eq.(\ref{eqn_tri_reci}) is convergent only if $\vert\varphi_{s,t}\vert>1$. If $\vert q\vert>1$, then the series in the left-hand of Eq.(\ref{eqn_tri_reci}) is convergent only if $\vert\varphi_{s,t}^\prime\vert>1$.
By partial fractions
    \begin{align*}
        \frac{1}{\brk[c]{n}_{s,t}\brk[c]{n+1}_{s,t}}=\frac{A}{\brk[c]{n}_{s,t}}+\frac{B}{\brk[c]{n+1}_{s,t}}=\frac{A}{\brk[c]{n}_{s,t}}+\frac{B}{\varphi_{s,t}\brk[c]{n}_{s,t}+\varphi_{s,t}^{\prime n}}
    \end{align*}
where $A=(-\varphi_{s,t}/t)^n$ and $B=(-\varphi_{s,t}/t)^{n+1}$. Then
\begin{align*}
    \sum_{n=1}^{\infty}\frac{1}{\brk[c]{n}_{s,t}\brk[c]{n+1}_{s,t}}&=\sum_{n=1}^{\infty}\left(\frac{(-\varphi_{s,t}/t)^{n}}{\brk[c]{n}_{s,t}}+\frac{t(-\varphi_{s,t}/t)^{n+1}}{\brk[c]{n+1}_{s,t}}\right)\\
    &=-\frac{\varphi_{s,t}}{t}+(1+t)\sum_{n=2}^{\infty}\frac{(-\varphi_{s,t}/t)^{n}}{\brk[c]{n}_{s,t}}\\
    &=-\frac{\varphi_{s,t}}{t}-(1+t)\left(\ln_{s,t}\left(1+\frac{\varphi_{s,t}}{t}\right)-\frac{\varphi_{s,t}}{t}\right)\\
    &=\varphi_{s,t}-(1+t)\ln\left(1+\frac{\varphi_{s,t}}{t}\right).
\end{align*}
\end{proof}
The function $\ln(1-x)$ is convergent for all $x\in(-\vert\varphi_{s,t}\vert,\vert\varphi_{s,t}\vert)$.
Some specializations:
\begin{align*}
    \sum_{n=1}^{\infty}\frac{1}{\F_{n}\F_{n+1}}&=\frac{1+\sqrt{5}}{2}-2\ln_{\F}\left(\frac{3+\sqrt{5}}{2}\right).\\
    \sum_{n=1}^{\infty}\frac{1}{\P_{n}\P_{n+1}}&=1+\sqrt{2}-2\ln_{\P}\left(2+\sqrt{2}\right).\\
    \sum_{n=1}^{\infty}\frac{1}{\J_{n}\J_{n+1}}&=2-3\ln_{\J}\left(2\right).\\
    \sum_{n=1}^{\infty}\frac{1}{\M_{n}\M_{n+1}}&=2+\ln_{\M}\left(0\right).
\end{align*}
A very important conclusion is that $\ln_{M}(0)$ is finite.

\begin{corollary}
    \begin{equation}
        \sum_{n=1}^{\infty}\frac{(-t)^{n}}{\brk[c]{n}_{s,t}\brk[c]{n+1}_{s,t}}=\frac{s+\sqrt{s^2+4t}}{2}.
    \end{equation}
\end{corollary}
\begin{proof}
If $t=-1$ in the above theorem, then
\begin{equation}
    \sum_{n=1}^{\infty}\frac{1}{\brk[c]{n}_{s,-1}\brk[c]{n+1}_{s,-1}}=\varphi_{s,-1}.
\end{equation}
Note that $\brk[c]{n}_{s,-1}=(i/\sqrt{t})^{n-1}\brk[c]{n}_{a,t}$, where $s=ia/\sqrt{t}$. Then
\begin{align*}
    \sum_{n=1}^{\infty}\frac{1}{\brk[c]{n}_{s,-1}\brk[c]{n+1}_{s,-1}}&=
    \sum_{n=1}^{\infty}\frac{1}{(i/\sqrt{t})^{2n-1}\brk[c]{n}_{a,t}\brk[c]{n+1}_{a,t}}\\
    &=\frac{i}{\sqrt{t}}\sum_{n=1}^{\infty}\frac{(-t)^n}{\brk[c]{n}_{a,t}\brk[c]{n+1}_{a,t}}=\frac{i}{\sqrt{t}}\varphi_{a,t}.
\end{align*}
The proof is completed.
\end{proof}
Some specialization:
\begin{align*}
    \sum_{n=1}^{\infty}\frac{(-1)^{n}}{F_{n}F_{n+1}}&=\frac{1+\sqrt{5}}{2}.\\
    \sum_{n=1}^{\infty}\frac{(-1)^{n}}{P_{n}P_{n+1}}&=1+\sqrt{2}.\\
    \sum_{n=1}^{\infty}\frac{(-2)^{n}}{J_{n}J_{n+1}}&=2.\\
    \sum_{n=1}^{\infty}\frac{2^{n}}{M_{n}M_{n+1}}&=2.\\
    \sum_{n=1}^{\infty}\frac{1}{U_{n-1}(t)U_{n}(t)}&=t+\sqrt{t^2-1}.
\end{align*}

\end{document}